\documentclass[12pt,a4paper,english,reqno]{amsart}
\usepackage[a4paper,footskip=1.5em]{geometry}
\usepackage{amsmath,amssymb,amsthm,mathtools,bm}
\usepackage[mathscr]{euscript}
\usepackage[usenames,dvipsnames]{color}
\usepackage{adjustbox,tikz,calc,graphics,babel,standalone}
\usepackage{subcaption}
\usepackage{csquotes,enumerate,verbatim}
\usepackage[final]{microtype}
\usepackage[numbers]{natbib}
\usetikzlibrary{shapes.misc,calc,intersections,patterns,decorations.pathreplacing}
\usepackage{hyperref}
\hypersetup{colorlinks=true,linkcolor=blue,citecolor=blue}

\theoremstyle{plain}
\newtheorem*{theorem*}{Theorem}
\newtheorem{theorem}{Theorem}[section]
\newtheorem{lemma}[theorem]{Lemma}

\newtheorem{proposition}[theorem]{Proposition}
\newtheorem*{claim*}{Claim}

\newtheorem{conjecture}[theorem]{Conjecture}
\newtheorem{problem}[theorem]{Problem}

\theoremstyle{remark}

\def\M{\mathcal{M}}
\def\S{\mathcal{S}}
\def\N{\mathbb{N}}
\def\Z{\mathbb{Z}}

\def\C{\mathscr}
\def\B{\mathbf}

\DeclareMathOperator\Deg{d}

\DeclareMathOperator\Hom{hom}

\let\emptyset\varnothing
\let\eps\varepsilon

\let\originalleft\left
\let\originalright\right
\renewcommand{\left}{\mathopen{}\mathclose\bgroup\originalleft}
\renewcommand{\right}{\aftergroup\egroup\originalright}

\makeatletter
\def\imod#1{\allowbreak\mkern10mu({\operator@font mod}\,\,#1)}
\makeatother

\def\l_p[#1]{(\log m)^{#1}}

\begin{document}

\title{The multiplication table problem for bipartite graphs}

\author{Bhargav Narayanan}
\address{Department of Pure Mathematics and Mathematical Statistics, University of Cambridge, Wilberforce Road, Cambridge CB3\thinspace0WB, UK}
\email{b.p.narayanan@dpmms.cam.ac.uk}

\author{Julian Sahasrabudhe}
\address{Department of Mathematical Sciences, University of Memphis, Memphis TN 38152, USA}
\email{julian.sahasra@gmail.com}

\author{Istv\'{a}n Tomon}
\address{Department of Pure Mathematics and Mathematical Statistics, University of Cambridge, Wilberforce Road, Cambridge CB3\thinspace0WB, UK}
\email{i.tomon@dpmms.cam.ac.uk}

\date{18 August 2014}

\begin{abstract}
We investigate the following generalisation of the `multiplication table problem' of Erd\H{o}s: given a bipartite graph with $m$ edges, how large is the set of sizes of its induced subgraphs? Erd\H{o}s's problem of estimating the number of distinct products $ab$ with $a,b \le n$ is precisely the problem under consideration when the graph in question is the complete bipartite graph $K_{n,n}$. In this note, we prove that the set of sizes of the induced subgraphs of any bipartite graph with $m$ edges contains $\Omega(m/\l_p[12])$ distinct elements.
\end{abstract}

\maketitle

\section{Introduction}
For a bipartite graph $G$, we define, writing $e(.)$ for the number of edges of a graph, its \emph{multiplication table $\M(G)$} by setting
\[ \M(G) = \{ e(H) : H \mbox{ is an induced subgraph of } G \};\]
our reasons for calling $\M(G)$ the multiplication table of $G$ will soon become evident. While the definition of $\M(G)$ above is meaningful for any graph $G$, we shall mainly be concerned with bipartite graphs in this note.

It seems likely that the bipartite graphs with the smallest multiplication tables are the complete bipartite graphs; writing $K_{n,n}$ for the complete bipartite graph between two disjoint sets of $n$ vertices, we conjecture the following.

\begin{conjecture}\label{mult14-c:main}
Let $n \in \N$ and suppose that $G$ is a bipartite graph with $e(G) = n^2$. Then $| \M(G) | \ge |\M(K_{n,n})|$.
\end{conjecture}
In this note, our aim is to prove Conjecture~\ref{mult14-c:main} in a weak quantitative form. It turns out that $|\M(K_{n,n})| = o(n^2)$; in fact, as was shown by Ford~\citep{Ford1}, there exists an absolute constant $\delta \approx 0.086$ such that $|\M(K_{n,n})| = n^2/(\log n)^{\delta+o(1)}$. Our main result, stated below, gives a comparable bound for an arbitrary bipartite graph.

\begin{theorem}\label{mult14-t:main}
If $G$ is a bipartite graph with $m$ edges, then
\[ |\M(G)| = \Omega \left( \frac{m}{\l_p[12]} \right). \]
\end{theorem}

Let us make a few remarks about Theorem~\ref{mult14-t:main}. First, in the light of our earlier remarks about complete bipartite graphs, it is clear that we cannot do away with the logarithmic factor in our result; indeed, we cannot replace the exponent $12$ in the statement of Theorem~\ref{mult14-t:main} by an exponent smaller than $\delta \approx 0.086$. Next, we should point out that it is trivial to prove that $|\M(G)| = \Omega(m^{1/2})$ for any graph $G$ with $m$ edges; to see this, note that any such graph either contains a vertex of degree $\Omega(m^{1/2})$ or an induced matching of size $\Omega(m^{1/2})$, and we are done in either case. If we do not insist that our graph is bipartite, this trivial argument can be seen to be essentially tight by considering, for example, the complete graph. Hence, to beat this trivial bound of $m^{1/2}$, it is necessary to exploit the fact that we are working with bipartite graphs. We urge the reader to pause for a moment and consider the question of beating this trivial lower bound of $m^{1/2}$ for bipartite graphs; while Theorem~\ref{mult14-t:main} improves on this bound considerably, we do not have a short proof of even a lower bound of, say, $m^{2/3}$. Finally, while it might be possible to refine our methods to prove a bound of the form say, $\Omega(m/\l_p[9])$, it seems unlikely that our proof can be adapted to prove Conjecture~\ref{mult14-c:main}, or for that matter, to even prove a bound of $\Omega(m/\log m)$; hence, we make no serious attempt to optimise the logarithmic factors in our proof.

Our main motivation for studying Conjecture~\ref{mult14-c:main} is because it is a natural combinatorial generalisation of a number theoretic problem, now known colloquially as the `multiplication table problem', posed by Erd\H{o}s~\citep{Erdos1} in 1955. For $n \in \N$, write $[n]$ for the set $\{1, \dots, n\}$ and $[n] \cdot [n]$ for the set of distinct products $ab$ with $a, b \le n$; the multiplication table problem is then simple to state: how large is $[n] \cdot [n]$? Observe that $[n] \cdot [n]$ is precisely the set of sizes of the induced subgraphs of $K_{n,n}$; consequently, the multiplication table problem can be rephrased as follows: how large is $|\M(K_{n,n})|$? In this paper, we generalise this question and ask how large $|\M(G)|$ is for an arbitrary bipartite graph $G$ with a prescribed number of edges.

The multiplication table problem has received a great deal of attention over the past five decades. Erd\H{o}s~\citep{Erdos1} showed, using the fact that almost all natural numbers less that $n$ have about $\log\log n$ distinct prime factors, that the cardinality of $[n] \cdot [n]$ is $o(n^2)$ as $n \to \infty$. Subsequently, better bounds were obtained, first by Erd\H{o}s~\citep{Erdos2} and then by Tenenbaum~\citep{Tenenbaum}. Despite its innocuous appearance, the multiplication table problem has been settled only recently; a deep result of Ford~\citep{Ford1} asserts that as $n \to \infty$,
\[  \big\lvert [n] \cdot [n]  \big\rvert = \Theta \left( \frac{n^2}{(\log{n})^\delta (\log{\log{n}})^{3/2}}\right), \]
where $\delta = 1 - (1+\log\log2)/\log{2} \approx 0.086$. For results about higher dimensional analogues of the multiplication table problem, we refer the reader to the papers of Koukoulopoulos~\citep{Koukoulopoulos1, Koukoulopoulos2}. Erd\H{o}s also posed a number of related number theoretic problems; see the book of Hall and Tenenbaum~\citep{erd_book} and the section on the statistical theory of divisors in the comprehensive survey of Ruzsa~\citep{erd_survey}.

The problem studied in this paper is also closely related to a number of combinatorial results about induced subgraph sizes that have been proved over the course of the last thirty years. Many of these questions and results about the sizes of induced subgraph arise from trying to better understand the structure of \emph{Ramsey graphs}; we discuss some of these problems below.

A subset of the vertices of a graph is said to be \emph{homogeneous} if it either induces a clique or an independent set; let us write $\Hom(G)$ for the size of the largest homogeneous set of vertices in a graph $G$. Alon and Bollob\'as~\citep{AB} (see also~\citep{EH}) proved that any graph without a large homogeneous set necessarily contains many distinct (non-isomorphic) induced subgraphs; in their proof, they distinguished between induced subgraphs using, amongst other parameters, their order and size. Subsequently, Erd\H{o}s, Faudree and S\'os~\citep{EFS1, EFS2} conjectured that for every $C >0$, there exists an $\eps = \eps(C)>0$ such that if $G$ is an $n$-vertex graph with $\Hom(G) \le C \log n$, then the number of distinct pairs $(x, y)$ such that $G$ has an induced subgraph on $x$ vertices inducing $y$ edges is at least $\eps n^{5/2}$. While this conjecture still remains unresolved, a number of partial results have been proved; see the papers of Axenovich and Balogh~\citep{Ax}, Alon and Kostochka~\cite{AK}, and Alon, Balogh, Kostochka and Samotij~\citep{ABKW} for the state of the art.

Another conjecture with a similar flavour, due to Erd\H{o}s and McKay~\citep{EFS1, EM}, also far from settled, asserts that for every $C >0$, there exists an $\eps=\eps(C)>0$ such that if $G$ is an $n$-vertex graph with $\Hom(G) \le C\log n$, then $G$ contains an induced subgraph with precisely $y$ edges for every integer $y$ between $0$ and $\eps n^2$; the best known bounds for this problem are due to Alon, Krivelevich and Sudakov~\citep{AKS}.

The sizes of the induced subgraphs of a random graph have also been investigated; we refer the reader to the paper of Calkin, Frieze and McKay~\cite{CFM} for details. Finally, let us also mention that the connection between the multiplication table problem and the sizes of induced subgraphs of complete bipartite graphs was exploited by the first author in~\citep{m_col} to construct `good' colourings for a Ramsey theoretic problem.

Returning to the question at hand, let us discuss, very briefly, one of the difficulties in proving Theorem~\ref{mult14-t:main}. Note that a proof of Theorem~\ref{mult14-t:main} should also establish that $|[n] \cdot [n]| = \Omega(n^2/(\log{n})^{12})$. To prove such a weak estimate for the size of $[n] \cdot [n]$ is in itself not difficult. One could use the prime number theorem to show that one has many distinct products of the form $ab$ in the set $[n] \cdot [n]$ where both $a$ and $b$ are prime. Alternatively, one could use the fact that the set $[n]$ has small `additive doubling' to conclude, using a beautiful theorem of Solymosi~\citep{Solymosi}, that $[n] \cdot [n]$ is large. These are however the only ways, to the best of our knowledge, of proving a reasonable lower bound for $|[n] \cdot [n]|$ without resorting to somewhat involved divisor estimates, and neither of these methods would appear to generalise easily to the setting of bipartite graphs. Hence, to prove Theorem~\ref{mult14-t:main}, we shall require, in addition to graph theoretic techniques, a few additive combinatorial and number theoretic tools; some of these might be of independent interest.

The rest of this paper is organised as follows. In Section~\ref{mult14-s:prelim}, we establish some notation and then prove some straightforward number theoretic estimates. We then describe our strategy for proving Theorem~\ref{mult14-t:main} in Section~\ref{mult14-s:strat}. After establishing a key partitioning lemma in Section~\ref{mult14-s:part}, we prove Theorem~\ref{mult14-t:main} in Section~\ref{mult14-s:proof}. We conclude by discussing some problems in Section~\ref{mult14-s:conc}.

\section{Preliminaries}\label{mult14-s:prelim}
In this section, we establish some notation and collect together some number theoretic estimates that we shall make use of when proving our main result.

\subsection{Notation}
Given a set $S$ and $r \in \N$, we write $S^{(r)}$ for the family of subsets of $S$ of cardinality $r$. Given $A, B \subset \Z$, we write $A + B$ and $A \cdot B$ respectively for the set of distinct sums $a+b$ and products $ab$ with $a \in A$ and $b \in B$.

It will help to have some notation in place for working with finite sequences. Given a sequence of integers $\B{a} = (a_i)_{i=1}^n$, we define $\S(\B{a})$, its \emph{set of sums}, by setting
\[ \S(\B{a}) =\left\{\sum_{i\in I}a_i : I\subset [n] \right\}.\]
Given two sequences of integers $\B{a}$ and $\B{b}$, we write $\S(\B{a}, \B{b})$ for the set of sums of the \emph{concatenation of $\B{a}$ and $\B{b}$}; equivalently, $\S(\B{a}, \B{b}) = \S(\B{a}) + \S(\B{b})$. We write $k \circ a$ to denote the sequence of length $k$ each of whose terms is $a$. So for example, the set $\S(k \circ a, l \circ b)$ consists of those integers which can be written as $ax + by$ for some $0 \le x \le k$ and $0 \le y \le l$.

Our conventions for asymptotic notation are largely standard; however, we feel obliged to point out that when we write, say $\Omega_k(.)$, we mean that the constant suppressed by the asymptotic notation is allowed to depend on (but is completely determined by) the parameter $k$. Occasionally, we shall find it convenient to switch to Vinogradov's notation: given functions $f$ and $g$, we write $f \ll g$ if $f = O(g)$ and $f \gg g$ if $g = O(f)$.

We use standard graph theoretic notation and refer the reader to~\citep{belabook1} for terms and notation not defined here.
To keep the exposition uncluttered, we omit floors and ceilings whenever they are not crucial.

\subsection{Number theoretic estimates} We now collect together a few easy number theoretic estimates; for the sake of completeness, we shall prove them.

\begin{lemma} \label{mult14-modulolemma}
Let $\B{a} = (a_i)_{i=1}^n$ be a sequence of positive integers and let $b \in \N$ be a positive integer such that $\gcd(a_i,b) \le g$ for $1 \le i \le n$. Then for any $k \in \N$,
\[ \S (\B{a}, k \circ b) \ge k\min\{b/g, n\}.\]
\end{lemma}

\begin{proof} Writing $\S_b (\B{a})$ for the set of residues modulo $b$ attained by the elements of $\S(\B{a})$, it is clearly sufficient to show that  $|\S_b (\B{a})| \ge \min\{b/g, n\}$; we shall prove this by induction on $n$.

The result is trivial if $n = 1$, so suppose that $n > 1$. Consider $\B{a}' = (a_i)_{i=1}^{n-1}$, and assume inductively that $|\S_b(\B{a}')| \ge \min\{b/g, n-1\}$. We are done if $|\S_b(\B{a}')| \ge n$ since then, $|\S_b(\B{a})| \ge |\S_b(\B{a}')| \ge n \ge \min\{b/g, n\}$. Also, if $n-1\ge b/g$, then we are done once again since $\min\{b/g, n\} = \min\{b/g, n-1\} = b/g$.

Hence, we may assume that $\S_b(\B{a}')$ contains exactly $n-1$ distinct residues modulo $b$ and also that $n-1 < b/g$. Let $t = \lceil b/g \rceil - 1$ and note that since $n-1 < b/g$, it is also true that $n-1 \le t$. Choose $s\in \S_b(\B{a}')$ and observe that the numbers
\[s, s+a_n, s+2a_n, \dots,s+t a_n\]
are all distinct modulo $b$ since $\gcd(a_n,b) \le g$. Also, as $|\S_b(\B{a}')| = n-1\le t$, one of
\[ s+a_n, s+2a_n, \dots,s+t a_n \]
is not in $\S_b(\B{a}')$ since these numbers are, modulo $b$, all distinct and distinct from $s \in \S_b(\B{a}')$. Now choose the minimal $1 \le l \le t$ such that $s+la_n\not\in \S_b(\B{a}')$. By the minimality of $l$, we have $s+(l-1)a_n\in \S_b(\B{a}')$, so $s+(l-1)a_n + a_n = s+la_n\in \S_b(\B{a})$. Consequently, $s+la_n\in \S_b(\B{a})\setminus \S_b(\B{a}')$ and we are done.
\end{proof}

Let us record here, for convenience, a special case of Lemma~\ref{mult14-modulolemma}
\begin{lemma}\label{mult14-numtheo}
Let $a,b\in \N$ be positive integers such that $\gcd(a,b) \le g$. Then for any $k, l \in \N$,
\[|\S(k \circ a, l \circ b)|\ge k\min\{ a/g ,l\}. \eqno\qed\]
\end{lemma}

We need the following easy consequence of the prime number theorem; see~\citep{pntbook}, for example.
\begin{proposition}\label{mult14-pnt}
For every $\eps > 0$, the number of primes in the interval $[n,(1+\eps)n]$ is $\Omega_\eps(n/\log n)$ as $n \to \infty$. \hfill \qed
\end{proposition}
Indeed, it follows from the prime number theorem that the number of primes in the interval $[n,(1+\eps)n]$ is asymptotic to $\eps n/\log n$ for any fixed $\eps > 0$. However, the weaker estimate above will be sufficient for our purposes.

\begin{lemma}\label{mult14-primes}
If $A \subset \N$ is a set of positive integers and $b \in \N$ is a positive integer such that $ \max{A} < b^{3}/16$, then
\[ \big\lvert A \cdot [b] \big\rvert \gg  \frac{|A|b}{\log b} .\]
\end{lemma}

\begin{proof} Let $P_b$ denote the set of primes in the interval $[b/2, b]$; we know that $|P_b| \gg b/\log b$ by Proposition~\ref{mult14-pnt}. Note that any $k \in A \cdot [b]$ has at most three distinct prime factors in $P_b$. Indeed, if not, then $k$ is divisible by at least four distinct primes each of which is at least $b/2$, whence $k \ge b^4/16$, which contradicts the fact that $k \le b \max A < b^4/16$. Consequently, the number of distinct ordered pairs $(a,p) \in A \times P_b$ such that $ap = k$ is at most $3 \times 2 = 6$. Hence,
\[ \big\lvert A \times [b] \big\rvert \ge \frac{|A||P_b|}{6} \gg  \frac{|A|b}{\log b} . \qedhere\]
\end{proof}

\begin{lemma}\label{mult14-gcdlemma}
Let $a,b, d, k \in \N$ be positive integers. Then
\[ \min_{0\le i< k} \gcd(a,b-id) \ll a^{1/k} d k^2.\]
Furthermore, if $a\neq b$, then
\[ \min_{0\le i< k} \gcd(a-id,b-id)\ll |a-b|^{1/k} d k^2.\]
\end{lemma}

\begin{proof} Let $g=\gcd(a,b,d)$ and $f_{i}=\gcd(a,b-id)$. We claim that
\[ \prod_{i = 0}^{k-1} f_i \mid a d^{k} \prod_{q<k}q^{\lceil k/q\rceil},\]
where in the above, $q$ ranges over the set of prime powers less than $k$. To check this claim, consider any prime $p$: it suffices to show that the largest power of $p$ dividing $\prod_{i = 0}^{k-1} f_i$ also divides $ad^{k} \prod_{q<k}q^{\lceil k/q\rceil}$. Given $p$, fix $0 \le j \le k-1$ so that the largest power of $p$ dividing $f_j$ is the greatest amongst $f_0, \dots, f_{k-1}$; since $f_j \mid a$, we have accounted for the contribution from $f_j$. Next, note that the largest power of $p$ dividing $f_i$ for any $i \neq j$ is the same as the largest power of $p$ dividing $\gcd(f_i, f_j)$. Observe that
\[ \gcd(f_{i},f_{j}) = \gcd(a, b-id, b-jd) =\gcd(a, b-id, (i-j)d),\]
whence it is clear that
\[ \gcd(f_{i},f_{j}) \mid \gcd(a, b-id, d) \gcd(a, b-id, i-j) \mid g (i-j) \mid d(i-j).\]
Consequently, the largest power of $p$ dividing $\prod_{i \neq j} f_i$ also divides $d^{k-1} \prod_{i \neq j} (i-j)$. It suffices to account for the largest power of $p$ dividing $\prod_{i \neq j} (i-j)$. But note that for any prime power $q<k$, the number of indices $i \neq j$ such that $q \mid (i-j)$ is at most $\lceil k/q\rceil$; the claim then follows.

An old result of Mertens asserts that
\[ \sum_{p<k} \frac{\log p} { p} = (1+o(1))\log k,\]
where in the above, $p$ ranges over the set of primes less than $k$. It follows that
\[ \prod_{q<k}q^{\lceil k/q\rceil} = k^{k+o(k)} \ll k^{2k}.\]
Hence, there exists an $ i\le k-1$ such that
\[ f_{i}\le a^{1/k}d\left(\prod_{q<k}q^{\lceil k/q\rceil}\right)^{1/k} \ll a^{1/k}dk^2. \]

To finish the proof of the lemma, note that the second assertion follows from the first since $\gcd(a-id,b-id) = \gcd(a-b,b-id)$.
\end{proof}

\section{Overview of our approach}\label{mult14-s:strat}
To illustrate our approach, we make an easy observation. Let $G = (X,Y;E)$ be a bipartite graph and suppose that both vertex classes of $G$ have the same size. Assume that $G$ is \emph{half-regular}; in other words, assume that the vertices on one side of the bipartition, say $X$, all have the same degree $d$. Assume also that there exists a vertex $v \in Y$ with $ n/\log n < \Deg(v) < n - n/\log{n}$, where $n = |X| = |Y|$. If we remove $v$ from $G$, then each vertex of $X$ has degree either $d$ or $d-1$ in the resulting graph; moreover, there are at least $n/\log{n}$ vertices with each of these two degrees. By considering induced subgraphs of the form $G[X' \cup Y \setminus \{ v \}]$ where $X' \subset X$, we see that $\S(k\circ d, l \circ (d-1)) \subset \M(G)$ for some pair of natural numbers $k, l \ge n/ \log n$. Since $d$ and $d-1$ are coprime, it follows from Lemma~\ref{mult14-numtheo} that
\begin{align*}
|\M(G)| &\ge | \S(k\circ d, l \circ (d-1)) | \\
&\ge \frac{n}{\log{n}} \min \left\{d,\frac{n}{\log n} \right\} \\
&\ge \frac{nd}{(\log n)^2} \ge \frac{m}{\l_p[2]}.
\end{align*}

To prove Theorem~\ref{mult14-t:main}, we first reduce the problem of bounding the size of the multiplication table of a general bipartite graph to the situation above, namely that of bounding the size of the multiplication table of a half-regular bipartite graph. We shall show (see Lemma~\ref{mult14-largepartition}) that given any bipartite graph $G=(X,Y;E)$, it is possible find a reasonably large subset of $X$ that can be partitioned into many groups $X_1, \dots ,X_k$ such that the sums of the degrees of the vertices in each of these groups is the same. If we now imagine contracting each such set $X_i$ into a single vertex, we obtain, after discarding $X \setminus (\bigcup_{i}X_i)$ from our graph, a half-regular bipartite \emph{multigraph}.

If we can then find a vertex in $Y$ whose degree is neither too large nor too small, then we finish the proof using a variant of the argument sketched above. However, it might be the case that there is no such vertex in $Y$. If $Y$ contains only a few vertices of very large degree, then our graph is somewhat sparse, and in this case, we use Lemma~\ref{mult14-graphlemma} to complete the proof. If it turns out that many vertices of $Y$ have very large degrees, then we show that there is a reasonably dense induced subgraph within which we can find many distinct subgraph sizes; the argument in this case is number theoretic in nature.

\section{A partitioning lemma}\label{mult14-s:part}
In this section, we shall prove a partitioning lemma for multisets of positive integers that will play an important part in the proof of the main result. This lemma allows us to select a large proportion of the vertices of a vertex class of a bipartite graph and partition the selected vertices into many small, disjoint sets of equal size and such that the sum of the degrees of the vertices in each of these sets is the same.

\begin{lemma}\label{mult14-largepartition}
Let $\B{a} = (a_i)_{i=1}^n$ be a sequence of positive integers with $\sum_{i=1}^na_i = m$. There exist positive integers $k, r, d \in \N$ and pairwise disjoint sets $I_{1},\dots,I_{k} \in [n]^{(r)}$ such that
\begin{enumerate}
\item $1\le r \le \log m$,
\item $kd \ge m/\l_p[3]$, and
\item $\sum_{i\in I_{j}}a_{i} = d$ for each $1\le j \le k$.
\end{enumerate}
\end{lemma}

Let us fix, for the rest of this section, $a_1, \dots, a_n$ and $m = \sum_{i=1}^na_i$. For $i \in [n]$, we shall think of $a_i$ as the \emph{weight on $i$}, and given a hypergraph $\C{F}$ on $[n]$, we define the \emph{weight covered by $\C{F}$} to be the sum of the weights of those vertices contained in at least one edge of $\C{F}$. Given $r,d \in \N$, it is natural to consider the $r$-uniform hypergraph $\C{F}(r,d)$ on $[n]$ whose edges are the sets $I \in [n]^{(r)}$ with $ \sum_{i \in I} a_i = d$. In this language, Lemma~\ref{mult14-largepartition} tells us that we can find appropriate $r, d \in \N$ so that the hypergraph $\C{F}(r,d)$ contains a matching, i.e., a set of independent edges, covering a $1/\l_p[3]$ proportion of the total weight, namely $m$.

Our first two propositions together show that one of the hypergraphs $\C{F}(r,d)$ must necessarily contain a matching with many edges; using this, we then show that one of these large matchings must cover a $1/\l_p[3]$ proportion of the total weight.

\begin{proposition}\label{mult14-p:match}
If $\C{F}(r, d)$ contains $l$ edges, then there exist positive integers $ r' \le r$ and  $d' \le d$ for which $\C{F}(r',d')$ contains a matching of size at least $l^{1/r}/r$.
\end{proposition}

\begin{proof} We prove the claim by induction on $r$. If $r=1$, then it clearly suffices to take $r' = r = 1$ and $d' = d$. Now suppose that $r>1$. If some $i \in [n]$ is contained in at least $l^{(r-1)/r}$ edges of $\C{F}(r,d)$, we proceed inductively as follows. By looking at the edges of $\C{F}(r,d)$ containing $i$, we see that $\C{F}(r-1,d-a_i)$ has at least $l^{(r-1)/r}$ edges. From the inductive hypothesis applied to $\C{F}(r-1,d-a_i)$, we conclude that there exist $r' \le r-1 < r$ and $d' \le d-a_i < d$ such that $\C{F}(r',d')$ contains a matching of size at least
\[ \frac{\left(l^{(r-1)/r}\right)^{1/r-1}}{r-1} = \frac{l^{1/r}}{r-1} \ge \frac{l^{1/r}}{r}.\]

Now suppose that each $i \in [n]$ is contained in at most $l^{(r-1)/r}$ edges of $\C{F}(r,d)$. In this case, we claim that $\C{F}(r,d)$ contains a matching of size $l^{1/r}/r$; in fact, we claim that any maximal matching of $\C{F}(r,d)$ contains at least $l^{1/r}/r$ edges. Indeed, if $\C{M}$ is a maximal matching of $\C{F}(r,d)$, then each edge of $\C{F}(r,d)$ meets at least one edge of $\C{M}$. However, each edge of $\C{M}$ meets at most $rl^{(r-1)/r}$ edges of $\C{F}(r,d)$. It follows that the number of edges in $\C{M}$ is at least $l / rl^{(r-1)/r}  =  l^{1/r}/r$.
\end{proof}

\begin{proposition}\label{mult14-mainlemma}
For any $r \in \N$, there exist positive integers $1 \le r' \le r$ and $d' \in \mathbb{N}$ for which $\C{F}(r',d')$ contains a matching of size at least $n/(r^2 m^{1/r})$.
\end{proposition}
\begin{proof}
As $a_1, \dots, a_n$ are positive integers whose sum is $m$, it is clear that $\sum_{i \in I}a_i \in [m]$ for all $I\in [n]^{(r)}$. Consequently, there exists a $d \in [m]$ for which $\C{F}(r,d)$ contains at least $\binom{n}{r}/m$ edges. It follows from Proposition~\ref{mult14-p:match} that there exist positive integers $r' \le r$ and $d' \le d$ for which $\C{F}(r',d')$ contains a matching of size at least
\[ \frac{1}{r} \left(\frac{\binom{n}{r}}{m} \right)^{1/r} \ge \frac{1}{r} \left(\frac{n^r}{r^rm} \right)^{1/r} = \frac{n}{r^2m^{1/r}}. \qedhere\]
\end{proof}

We are now ready to prove Lemma~\ref{mult14-largepartition}.

\begin{proof}[Proof of Lemma~\ref{mult14-largepartition}]
To prove the lemma, we greedily apply Proposition~\ref{mult14-mainlemma} so as to cover $[n]$ with $\l_p[3]$ matchings, where each matching is from one of the hypergraphs $\C{F}(r,d)$. We can then conclude that one of these matchings covers a $1 / \l_p[3]$ proportion of the total weight, whence follows the lemma.

Define a collection of sets $A_{1} \supset A_{2}\supset  \dots $ recursively as follows. First set $A_1 = [n]$. Assume that we have defined $A_{t} \subset [n]$, the set of uncovered points after $t$ steps. We know from Proposition~\ref{mult14-mainlemma} that we can find $1\le r_t \le \log m$, $k_{t} \ge |A_{t}|/\l_p[2]$ and pairwise disjoint sets $I_{t,1},\dots,I_{t,k_{t}}\in A_{t}^{(r_{t})}$ such that, for some $d_t \in \N$, $\sum_{i\in I_{t,j}}a_{i} = d_{t}$ for each $1 \le j \le k_{t}$. Now define
\[ A_{t+1}=A_{t}\setminus \bigcup_{j=1}^{k_{t}}I_{t,j}. \]

We claim that it takes at most $\l_p[3]$ steps before we cover all the elements of $[n]$. To see this, simply note that at stage $t+1$, the number of elements we remove from $A_t$ to form $A_{t+1}$ is $k_tr_t \ge k_t \ge |A_{t}|/\l_p[2]$ and hence,
\[ |A_{t+1}| \le \left(1 - \frac{1}{\l_p[2]} \right)^tn < e^{ -t/\l_p[2]}n.\]

Since $n \le m$, it follows that $A_{\l_p[3]+1} = \emptyset$. It follows that there exists a $t \le \l_p[3]$ for which the weight covered by the matching $I_{t,1},\dots,I_{t,k_{t}}$ is at least $m/\l_p[3]$, thus concluding the proof of the lemma.
\end{proof}

\section{Proof of the main result}\label{mult14-s:proof}
As we remarked in Section~\ref{mult14-s:strat}, to prove Theorem~\ref{mult14-t:main}, we shall find it useful to work with multigraphs. Before we proceed further, we set out some notation.

Let $G=(V,E)$ be a multigraph. We say that two vertices are neighbours in $G$, or are adjacent to each other in $G$, if they are joined by \emph{at least one edge} in $G$. Given $v \in V$, we write $\Deg(v)$ for the number of edges incident to $v$, and $\Gamma(v)$ for the set of neighbours of $v$; as $\Deg(v)$ is not necessarily equal to the cardinality of $\Gamma(v)$ in a multigraph, we write $\gamma(v) = |\Gamma(v)|$ for the number of distinct neighbours of $v$. Given $v \in V$ and $t\in \N$, we write $\Gamma_{t}(v)$ for the set of vertices joined to $v$ by exactly $t$ parallel edges and define $\gamma_t(v) = |\Gamma_t(v)|$. So for example, $\Gamma_0(v)$ is the set of vertices not adjacent to $v$, $\Gamma(v) = \bigcup_{t \ge 1} \Gamma_t (v)$, $\gamma(v) = \sum_{t\ge 1}\gamma_t(v)$, and $\Deg(v) =  \sum_{t \ge 0} t \gamma_t (v)$.

Given $U \subset V$, we denote the subgraph of $G$ induced by $U$ by $G[U]$. For a subgraph $H$ of $G$, we shall write $e(H)$ for the number of edges of $H$ counted with multiplicity, and for $U \subset V$, we write $e(U)$ to denote $e(G[U])$. Finally, we define the multiplication table $\M(G)$ of $G$ as we did for simple graphs by setting
\[\M(G)=\{e(U): U\subset V\}.\]

If $G=(X,Y;E)$ is a bipartite graph with vertex classes $X$ and $Y$, then note that for any $X' \subset X$ and $Y' \subset Y$, we have
\[ e(X' \cup Y') = \sum_{x \in X'} \Deg(x) - e(X' \cup (Y \setminus Y')); \]
we shall make use of this simple observation repeatedly in the proof of Theorem~\ref{mult14-t:main}.

The next lemma will be useful when dealing with sparse graphs in the proof of Theorem~\ref{mult14-t:main}; an analogous proposition for \emph{simple} bipartite graphs appears in~\citep{AKS}.

\begin{lemma}\label{mult14-graphlemma}
Let $G=(X,Y;E)$ be a bipartite multigraph with at most $r\ge1$ parallel edges between any pair of vertices and suppose that each vertex of $G$ has positive degree. If $|X| = n$, then
\[|\M(G)\cap [l]| \ge \frac{l}{2r}\]
for each $1\le l\le n$.
\end{lemma}

\begin{proof} If a vertex $y \in Y$ is such that each of its neighbours in $X$ has two or more distinct neighbours in $Y$, we delete $y$ from $Y$. Doing this repeatedly if necessary, we assume that every $y \in Y$ has a neighbour $x_y \in X$ with $\Gamma(x_y)=\{ y \}$. Note that even after these deletions, every vertex still has positive degree.

First, suppose that there exists a $y\in Y$ with $\gamma(y)>n/2r$. Then if $ X_1 \subset \dots \subset X_{l/2r}\subset \Gamma(y)$ is a strictly monotone increasing sequence of sets with $|X_{i}|=i$, then we see that the sizes of the subgraphs induced by the sets $X_i \cup \{y\}$ are all distinct and contained in $[l]$.

Hence, suppose that $\gamma(y)\le n/2r$ for every $y\in Y$. We shall construct a sequence $X_{0}, \dots ,X_{n/2r}$ of subsets of $X$ with $|X_{i}|\le i$ with the property that the sequence $(e_i)_{i=0}^{n/2r}$ defined by $e_i = e(X_i \cup Y)$ satisfies $0 < e_{i+1}-e_{i} \le r$ for each $i \ge 0$. If we can do this, then we are done since
\[e_{1},\dots,e_{l/2r}\in \M(G)\cap [l]\]
for each $l \in [n]$.

We build the sets $X_i$ recursively. We begin by setting $X_0 = \emptyset$. Having constructed $X_{i}$, we construct $X_{i+1}$ as follows. If there exists a $y\in Y$ with $x_{y} \not \in X_{i}$, we take $X_{i+1}=X_{i}\cup \{x_{y}\}$ in which case it is clear that $|X_{i+1} |\le i+1$; we also have $0 < e_{i+1}-e_{i} \le r$ since $y$ and $x_y$ are joined by at least 1 and at most $r$ edges.

Now suppose that $x_y \in X_i$ for every $y \in Y$. Since we have assumed that $\gamma(y)\le n/2r$ for each $y\in Y$, and since any two vertices of $G$ are joined by at most $r$ parallel edges, it follows that $e(G) \le n|Y|/2$. Since $|X_i| \le i \le n/2$, we conclude by double counting that there is a vertex $x \in X \setminus X_i$ such that $\Deg(x) \le |Y|$. Choose some $k$ vertices $y_1, \dots, y_k \in Y$ so that
\[ 0 < \Deg(x) - \sum_{j=1}^k \Deg(x_{y_j}) \le r.\]
This is possible since $\sum_{y \in Y}\Deg(x_{y}) \ge |Y| \ge \Deg(x)$ and since, for each $y \in Y$, $\Deg(x_y) \le r$ as $y$ is the sole neighbour of $x_y$. Now define
\[ X_{i+1} = \left(X_i \setminus \bigcup_{j=1}^kx_{y_j}\right) \cup \{x\}. \]
Clearly, $|X_{i+1}| \le i+1$. Since we have assumed that $x_y \in X_i$ for each $y \in Y$, we also see that
\[ e_{i+1} = e_i + \Deg(x) - \sum_{j=1}^k \Deg(x_{y_j}).\]

We have shown how to construct $X_{i+1}$ with all the requisite properties; this completes the proof.
\end{proof}

We are now in a position to prove Theorem~\ref{mult14-t:main}.
\begin{proof}[Proof of Theorem~\ref{mult14-t:main}]
Let $G = (X, Y; E)$ be a bipartite graph with $m$ edges; all the inequalities in our proof will hold when $m$ is sufficiently large.

The first step in proving Theorem~\ref{mult14-t:main} is to pass from $G$ to an induced subgraph of $G$ within which we have better control over the vertex degrees, while at the same time retaining a large fraction of the edges of $G$ in this induced subgraph.

Let $|X| = n_1$ and let $x_1, \dots, x_{n_1}$ be the vertices of $X$. Applying Lemma~\ref{mult14-largepartition} to the sequence $\Deg(x_1), \dots, \Deg(x_{n_1})$, we see that it is possible to find positive integers $k_1, r_1, d_1 \in \N$ and pairwise disjoint sets $X_1, \dots, X_{k_1} \in X^{(r_1)}$ such that $r_1 \le \log m$, $k_1 d_1 \ge m/\l_p[3]$, and

\[ \sum_{x \in X_j} \Deg(x) = d_1 \]
for $1 \le j \le k_1$. We delete the vertices $X \setminus (\bigcup_{j = 1}^{k_1} X_j)$ from $G$ and also discard any vertices of $Y$ which subsequently become isolated; note that our graph still has $k_1d_1 \ge m/\l_p[3]$ edges.

If $k_1$ is small, we need to work a bit harder. Suppose that $k_1 < m^{1/2}/ \l_p[4]$. Let $|Y| = n_2$ and let $y_1, \dots, y_{n_2}$ be the vertices of $Y$. We again apply Lemma~\ref{mult14-largepartition}, but on this occasion to the sequence $\Deg(y_1), \dots, \Deg(y_{n_2})$, to find positive integers $k_2, r_2, d_2 \in \N$ and pairwise disjoint sets $Y_1, \dots, Y_{k_2} \in Y^{(r_2)}$ such that $r_2 \le \log m$, $k_2d_2 \gg m/\l_p[6]$, and
\[ \sum_{y \in Y_j} \Deg(y) = d_2 \]
for $1 \le j \le k_2$. We then delete the vertices $Y \setminus (\bigcup_{j = 1}^{k_2} Y_j)$ from $G$ and discard any vertices of $X$ which subsequently become isolated. Observe that $G$ still has $k_2d_2\gg m/\l_p[6]$ edges.

Notice that after these deletions, $|X| \le k_1r_1 \le k_1\log{m}$ and $|Y| = k_2r_2 \le k_2\log{m}$. Since $G$ still has $\Omega( m/\l_p[6])$ edges, it follows that
\[ k_1 k_2\l_p[2] \gg \frac{m}{\l_p[6] }. \]
Consequently, if $k_1 < m^{1/2}/ \l_p[4]$, then $k_2 \gg (m^{1/2}/ \l_p[4]$.

Relabelling $X$ and $Y$ if necessary, note that $G$ now has the property that that there exist positive integers $k, r, d \in \N$ such that $k \gg m^{1/2}/\l_p[4]$, $r \le \log m$, $e(G) = kd \gg m/\l_p[6]$, and there exists a partition
\[ X = \bigcup_{j = 1}^k X_j \]
of $X$ into $k$ sets each of cardinality $r$ with the property that
\[ \sum_{x \in X_j} \Deg(x) = d \]
for $1 \le j \le k$.

Let $H = (X_H, Y_H; E_H)$ be the half-regular bipartite multigraph obtained from $G$ by contracting the vertices of each $X_j$ into a single vertex for $1 \le j \le k$. Clearly,
\[\M(H) \subset \M(G),\]
so it suffices to bound $|\M(H)|$ from below. Henceforth, we shall work exclusively with $H$, so in what follows, all vertex degrees, neighbourhoods, etc.\ will be with respect to the multigraph $H$. For easy reference, let us list the properties of $H$ that we shall require in the rest of the proof.

\begin{enumerate}
\item $H$ has no isolated vertices.
\item There are at most $r \le \log{m}$ parallel edges between any two vertices of $H$.
\item For each $x \in X_H$, $\Deg(x) = d$.
\item $|X_H| = k \gg m^{1/2}/\l_p[4]$.
\item $e(H) = kd \gg m/\l_p[6]$.
\end{enumerate}

Our goal now is to establish that $|\M(H)| \gg m/\l_p[12]$. We may assume that $d \ge \l_p[4]$. If not, then $k \gg m/\l_p[10]$ since $kd \gg m/\l_p[6]$. But then $|\M(H)| \ge k  \gg m/\l_p[10]$ since clearly, $d, 2d, \dots, kd \in \M(H)$.

We claim that we are also done if there exists a vertex $y \in Y_H$ and $0 \le a < b \le r$ such that $\gamma_a(y), \gamma_b(y) \ge k/2\l_p[2]$. Indeed, if such a $y$ exists, choose $V_1 \subset \Gamma_a (y)$ and $V_2 \subset \Gamma_b (y)$ and note that
\[ e(V_1 \cup V_2 \cup (Y_H \setminus \{ y\})) = (d-a)|V_1| + (d-b)|V_2|,\]
from which it follows that
\[ \S\left(\gamma_a (y) \circ (d-a) , \gamma_b (y) \circ (d-b)\right)  \subset \M(H).\] Note that $\gcd(d-a, d-b) \le b-a \le \log m$. Since $k \gg m^{1/2}/\l_p[4]$, $d\ge \l_p[4]$ and $kd \gg m/\l_p[6]$, it follows from Lemma~\ref{mult14-numtheo} that
\begin{align*}
|\M(H)| &\ge \frac{k}{2\l_p[2]}\min \left\{ \frac{k}{2\l_p[2]} , \frac{d - \log m}{\log m}\right\}\\
&\gg \min \left\{ \frac{k^2}{\l_p[4]} , \frac{kd}{\l_p[3]}\right\}\\
&\gg \min \left\{ \frac{m}{\l_p[12]} , \frac{m}{\l_p[9]}\right\} \gg \frac{m}{\l_p[12]}.
\end{align*}

Hence, in what follows, we shall assume that $d \ge \l_p[4]$ and that for each $y \in Y_H$, there is at most one $0 \le \tau \le r$ such that $\gamma_\tau(y) \ge k/2\l_p[2]$.

Since $\sum_{i=0}^r \gamma_i(y) = k$ and $r \le \log m$, a consequence of assuming there is at most one $0 \le \tau \le r$ for which $\gamma_\tau(y) \ge k/2\l_p[2]$ is that there in fact exists a unique $\tau$ for which $\gamma_\tau(y) \ge k(1-1/2\log m)$; we call this unique value $\tau$ the \emph{type of $y$} and say that the vertex $y$ is of \emph{type-$\tau$}. In the rest of the proof, each $y\in Y_H$ will be assumed to have unique type $\tau \le \log m$.

Note that a vertex $y \in Y_H$ of type-$0$ only has a few distinct neighbours in $X_H$. We shall distinguish two cases depending on the number of type-$0$ vertices of $Y_H$. We first deal with the case where most vertices are of type-$0$.

\textbf{Case 1: All but at most $d/2\log m$ vertices of $Y_H$ are of type-$0$.} Note that since each $x \in X_H$ has at least $d/r \ge d/\log m$ distinct neighbours in $Y_H$, each $x \in X_H$ is adjacent to at least one vertex of type-$0$ in $Y_H$ as  at most $d/2\log m$ vertices of $Y_H$ are of nonzero type. Next, observe that if $y\in Y_H$ is of type-$0$, then
\[\gamma(y) = \sum_{t=1}^{r} \gamma_t(y) < \frac{rk}{2\l_p[2]} \le \frac{k}{2\log m}.\]
 Consequently, we can greedily construct a set $U \subset Y_H$ of type-$0$ vertices such that the set
 \[S = \bigcup_{y \in U} \Gamma(y) \subset X_H\] satisfies, writing $s = |S|$,
 \[  \frac{k}{2}-\frac{k}{2\log{m}} \le s \le \frac{k}{2} .\]

Let $F$ be the subgraph of $H$ induced by $S \cup U$. We conclude from Lemma~\ref{mult14-graphlemma} that
\[ \M(F) \cap [d] \ge \min \left\{\frac{k}{6\log m}, \frac{d}{2\log m}\right\}.\]
The inequality above follows directly from Lemma~\ref{mult14-graphlemma} when $s \ge d$. If $s < d$, then since $s\ge k/2 - k/2\log m > k/3$, we have $\M(F) \cap [k/3] \subset \M(F) \cap [d]$ and the claimed inequality once again follows from Lemma~\ref{mult14-graphlemma}. We conclude that $\M(F)$ contains $\Omega (\min\{k, d\}/\log m)$ different values modulo $d$.

Now, if $S' \subset S$, $X' \subset X_H \setminus S$ and $U' \subset U$, then
\[ e(S'\cup X'\cup (Y_H\setminus U') ) = d|X'|+d|S'|-e(S'\cup  U'). \]
As we have already observed, there exist $\Omega (\min\{k, d\}/\log m)$ choices of $S'$ and $U'$ for which the quantities $d|S'|-e(S'\cup  U')$ are all distinct modulo $d$; since $|X'|$ can be any integer between $0$ and $k - s \ge k/2$, we see that
\[|\M(H)| \gg k \min\left \{ \frac{k}{\log m},\frac{d}{\log m} \right \} = \min\left \{ \frac{k^2}{\log m},\frac{kd}{\log m} \right \} \gg \frac{m}{\l_p[9]}.\]

\textbf{Case 2: At least $d/2\log m$ vertices of $Y_H$ are not of type-$0$.}
Set $p = \log m$ and $q = d/4\log m$. As $d \ge \l_p[4]$, we can find $1 \le \tau \le \log m$ and a set $U \subset Y_H$ of $p$ vertices all of which are of type-$\tau$. Let
\[ S=\bigcap_{y \in U} \Gamma_\tau(y) \subset X_H;\]
since $\gamma_\tau(y) \ge k(1-1/2\log m)$ for each $y \in U$, note that $s=|S| \ge k/2$.

Now choose $q$ vertices from $Y_H \setminus U$, say $y_{1},\dots,y_{q}$, so that each of these vertices is of nonzero type; this is possible since there are at least $d/(2\log m) - p > q$ vertices of nonzero type in $Y_H \setminus U$. For $1 \le i \le q$, let the type of $y_i$ be $\tau_i>0$, let $Y_{i}=Y_H \setminus\{y_{1},\dots,y_{i}\}$ and let $H_{i}$ be the subgraph of $H$ induced by $S\cup Y_{i}$ . We say that $H_i$ is \emph{good} if at least $2s/3$ vertices of $S$ have the same degree in $H_{i}$. We take $H_{0} = H$; clearly, $H_0$ is good since every vertex of $S$ has degree $d$ in $H_{0}$.

\textbf{Case 2A: $H_{1},\dots,H_{q}$ are all good.}
Since $H_i$ is good, we know that there are at least $2s/3$ vertices of $S$ with the same degree in $H_i$; let $\alpha_i$ be this common degree and let $S_{i}\subset S$ be the set of those vertices with degree $\alpha_i$ in $H_i$. Clearly, $\alpha_0 = d$. We claim that $\alpha_{i}=\alpha_{i-1}-\tau_{i}<\alpha_{i-1}$ for each $1 \le i \le q$. To see this, first note that every vertex of $S_{i-1}\cap \Gamma_{\tau_{i}}(y_{i})$ has degree $\alpha_{i-1}-\tau_{i}$ in $H_{i}$. Recall that $s = |S| \ge k/2 \gg m^{1/2}/\l_p[4]$ and $\gamma_{\tau_i}(y_{i})\ge k(1-1/2\log m)$; since $H_{i-1}$ is good,
\[ \lvert S_{i-1}\cap \Gamma_{\tau_{i}}(y_{i})\rvert \ge |S_{i-1}|-\frac{k}{2\log m} \ge \frac{2s}{3} - \frac{k}{2\log{m}} > \frac{s}{2}.\]
As we have assumed that $H_i$ is good, we know that if more than $s/2$ vertices of $S$ have the same degree in $H_i$, then these vertices must all belong to $S_i$ and hence, $\alpha_i = \alpha_{i-1}-\tau_{i}$.

If $S'_i \subset S_{i}$, then note the $e(S'_i \cup Y_i) = \alpha_i|S'_i|$. Hence, writing $A = \{\alpha_1, \dots, \alpha_q\}$, we see that
\[ A \cdot [2s/3] \subset \M(H) \]
Clearly, $\max A \le d \le m$ while $s^3 \gg k^3 \gg m^{3/2}/ \l_p[12]$, so by Lemma~\ref{mult14-primes}, $|A \cdot [2s/3]| \gg sq/\log s$. It follows that
\[ |\M(H)| \gg \frac{sq}{\log s} \gg \frac{kd}{\l_p[2]} \gg \frac{m}{\l_p[8]}. \]

\textbf{Case 2B: One of $H_1, \dots, H_q$ is not good.}
Let $1 \le l \le q$ be the minimal index for which $H_l$ is not good. Since $H_0, \dots, H_{l-1}$ are all good, we can, arguing as in the previous case, find $s/2$ vertices of $S$ which all have the same degree $\alpha$ in $H_l$ with $\alpha \ge d - l\log m \ge d - q\log m = 3d/4$; let $S_\alpha \subset S$ be this set of vertices. Also, as $H_l$ is not good, we know that $S_\beta = S \setminus S_\alpha$ contains at least $s/3$ vertices.

For $x \in S_\beta$, let $\beta_x \neq \alpha$ denote the degree of $x$ in $H_l$. For each $x \in S_\beta$, there exists, by Lemma~\ref{mult14-gcdlemma}, an $0\leq f_x\leq \log m$  such that
\begin{align*}
\gcd(\alpha-f_{x}\tau,\beta_{x}-f_{x}\tau) &\ll |\alpha - \beta_x|^{1/\log m} \tau \l_p[2] \\
&\ll m^{1/\log m} \tau \l_p[2] \ll \l_p[3].
\end{align*}
So there exists $0 \leq f\leq \log m$ and a set $S_f \subset S_\beta$ of size at least at least $|S_\beta |/\log m \ge s/3 \log m \ge k/ 6 \log m$ such that for each $x \in S_f$,
\[\gcd(\alpha-f\tau,\beta_{x}-f\tau) \ll \l_p[3].\]

Recall that $p = \log m \ge f$, and that $S=\bigcap_{y \in U} \Gamma_\tau(y)$, where $U\subset Y_H$ is a set of at least $p$  vertices of type-$\tau$. Fix a subset of $U$ of size $f$, say $U_f$. We shall only consider the induced subgraphs of $H[S\cup (Y_l\setminus U_f)]$. If $S_\alpha' \subset S_{\alpha}$ and $S_f' \subset S_f$, then note that
\[ e(S_\alpha' \cup S_f' \cup (Y_{l} \setminus U_f )) = (\alpha-f\tau)|S'_\alpha| + \sum_{x\in S_f'}(\beta_x - f\tau). \]
Hence, writing $\boldsymbol \beta$ for the sequence $(\beta_x - f\tau)_{x \in S_f}$, we see from the arguments above that
\[ \S(\boldsymbol \beta, s/2 \circ (\alpha - f\tau)) \subset \M(H)\]
since $|S_\alpha| \ge s/2$.

As $\gcd(\alpha-f\tau,\beta_{x}-f\tau) \ll \l_p[3]$ for each $x\in S_f$ and since $\alpha-f\tau \ge 3d/4 - \l_p[2] > d/2$, we use Lemma~\ref{mult14-modulolemma} to deduce that
\begin{align*}
|\M(H)| &\ge |\S(\boldsymbol \beta, s/2 \circ (\alpha - f\tau))| \\
&\gg \frac{s}{2} \min \left\{ \frac{\alpha - f\tau}{\l_p[3]},|S_f| \right\}\\
&\gg \frac{k}{4} \min \left\{ \frac{d}{\l_p[3]},\frac{k}{\log m} \right\}\\
&\gg \min \left\{ \frac{kd}{\l_p[3]},\frac{k^2}{\log m} \right\} \gg \frac{m}{\l_p[9]}.
\end{align*}

This concludes the proof of Theorem~\ref{mult14-t:main}.
\end{proof}

\section{Conclusion}\label{mult14-s:conc}
There are a number of problems related to the question studied in this paper worth investigating of which Conjecture~\ref{mult14-c:main} is perhaps the most natural. We discuss a few other related questions below.

Let $\M(m)$ denote the minimum value of $|\M(G)|$ taken over all bipartite graphs $G$ with $m$ edges. Trivially, $\M(m) \le m$, and in this note, we have shown that $\M(m) \gg m/\l_p[12]$. The question of determining the correct order of magnitude of $\M(m)$ still remains.

\begin{problem}\label{mult14-p:m}
Determine the asymptotic order of magnitude of $\M(m)$.
\end{problem}

We suspect Problem~\ref{mult14-p:m} might be difficult. For example, it is not at all clear that $\M(m)$ is an increasing function; indeed, we believe otherwise. We propose the following question as a possible first step towards settling Problem~\ref{mult14-p:m}.

\begin{problem}
Is $\M(m) = o(m)$ for every $m \in \N$?
\end{problem}

We know that $\M(n^2)=o(n^2)$ and in general, we suspect that the exact value of $\M(m)$ depends a great deal on how close $m$ is to a number with a reasonably `balanced' factorisation. Let us say that $m$ is \emph{$k$-balanced} if there exist positive integers $a,b \ge k$ such that $m = ab$. The set of positive integers $m$ such that $m$ is $(\log m)$-balanced has asymptotic density $1$ in $\N$. If $m = ab$ is a $(\log m)$-balanced factorisation of $m$, then as noted in~\citep{m_col}, one can show that $\M(m) = o(m)$ by considering $K_{a,b}$, the complete bipartite graph between two disjoint sets of size $a$ and $b$, and using Ford's estimates for the size of the set $[a] \cdot [b]$. It would be interesting to decide if one can say something similar for all sufficiently large positive integers.

Finally, it would be interesting to determine the structure of extremal graphs. Recall that Conjecture~\ref{mult14-c:main} asserts that the amongst all bipartite graphs with $n^2$ edges, $K_{n,n}$ has the smallest multiplication table. Of course one could, and should, ask what the extremal graphs are when the number of edges is no longer a square. We believe that if $|\M(G_m)|=\M(m)$ for some bipartite graph $G_m$ with $m$ edges, then $G_m$ must necessarily contain a large subgraph that `resembles' a complete bipartite graph. For example, a natural conjecture is that $\M(n(n+1)) = |\M(K_{n,n+1})|$; in general however, we have no precise guesses for what the extremal graphs are.

\section*{Acknowledgements}
The problem of proving Conjecture~\ref{mult14-c:main} in a weak quantitative form was proposed by the first author in August 2013 at the 5\textsuperscript{th} Eml\'ekt\'abla Workshop in Budapest; also, some of the ideas used here to prove Theorem~\ref{mult14-t:main} were developed while the authors were visitors at the IMT Institute for Advanced Studies Lucca. The authors are grateful to Guido Caldarelli and the other members of the Complex Networks Group at IMT Lucca for their hospitality and in addition, the first author is grateful to the organisers of the Eml\'ekt\'abla Workshop for their hospitality.

\bibliographystyle{amsplain}
\bibliography{bipartite_mult}
\end{document}